\documentclass[reqno]{amsart}
\usepackage{amssymb}
\usepackage{eucal}
\usepackage{amsmath}
\usepackage{amscd}
\usepackage[dvips]{color}
\usepackage{multicol}
\usepackage[all]{xy}           
\usepackage{graphicx}
\usepackage{color}
\usepackage{colordvi}
\usepackage{xspace}

\usepackage[active]{srcltx} 

\usepackage{lipsum}

\newcommand{\nc}{\newcommand}
\newcommand{\delete}[1]{}
\nc{\mfootnote}[1]{\footnote{#1}} 
\nc{\todo}[1]{\tred{To do:} #1}

\delete{
\nc{\mlabel}[1]{\label{#1}}  
\nc{\mcite}[1]{\cite{#1}}  
\nc{\mref}[1]{\ref{#1}}  
\nc{\mbibitem}[1]{\bibitem{#1}} 
}

\nc{\mlabel}[1]{\label{#1}  
{\hfill \hspace{1cm}{\bf{{\ }\hfill(#1)}}}}
\nc{\mcite}[1]{\cite{#1}{{\bf{{\ }(#1)}}}}  
\nc{\mref}[1]{\ref{#1}{{\bf{{\ }(#1)}}}}  
\nc{\mbibitem}[1]{\bibitem[\bf #1]{#1}} 

\newtheorem{theorem}{Theorem}[section]
\newtheorem{definition}[theorem]{Definition}
\newtheorem{lemma}[theorem]{Lemma}

\newtheorem{prop-def}[theorem]{Proposition-Definition}

\newtheorem{remark}[theorem]{Remark}

\newtheorem{example}[theorem]{Example}

\newtheorem{proposition}[theorem]{Proposition}

\nc{\tred}[1]{\textcolor{red}{#1}}
\nc{\tblue}[1]{\textcolor{blue}{#1}}
\nc{\tgreen}[1]{\textcolor{green}{#1}}
\nc{\tpurple}[1]{\textcolor{purple}{#1}}
\nc{\btred}[1]{\textcolor{red}{\bf #1}}
\nc{\btblue}[1]{\textcolor{blue}{\bf #1}}
\nc{\btgreen}[1]{\textcolor{green}{\bf #1}}
\nc{\btpurple}[1]{\textcolor{purple}{\bf #1}}

\nc{\li}[1]{\textcolor{red}{Xiaomin:#1}}
\nc{\cm}[1]{\textcolor{blue}{Chengming: #1}}

\nc{\twovec}[2]{\left(\begin{array}{c} #1 \\ #2\end{array} \right )}
\nc{\threevec}[3]{\left(\begin{array}{c} #1 \\ #2 \\ #3 \end{array}\right )}
\nc{\twomatrix}[4]{\left(\begin{array}{cc} #1 & #2\\ #3 & #4 \end{array} \right)}
\nc{\threematrix}[9]{{\left(\begin{matrix} #1 & #2 & #3\\ #4 & #5 & #6 \\ #7 & #8 & #9 \end{matrix} \right)}}
\nc{\twodet}[4]{\left|\begin{array}{cc} #1 & #2\\ #3 & #4 \end{array} \right|}

\nc{\rk}{\mathrm{r}}

\nc{\gensp}{V} 
\nc{\relsp}{\Lambda} 
\nc{\leafsp}{X}    
\nc{\treesp}{\overline{\calt}} 

\nc{\vin}{{\mathrm Vin}}    
\nc{\lin}{{\mathrm Lin}}    

\nc{\gop}{{\,\omega\,}}     
\nc{\gopb}{{\,\nu\,}}
\nc{\svec}[2]{{\tiny\left(\begin{matrix}#1\\
#2\end{matrix}\right)\,}}  
\nc{\ssvec}[2]{{\tiny\left(\begin{matrix}#1\\
#2\end{matrix}\right)\,}} 

\nc{\su}{\mathrm{Su}}
\nc{\tsu}{\mathrm{TSu}}
\nc{\TSu}{\mathrm{TSu}}
\nc{\eval}[1]{{#1}_{\big|D}}
\nc{\oto}{\leftrightarrow}

\nc{\oaset}{\mathbf{O}^{\rm alg}}
\nc{\omset}{\mathbf{O}^{\rm mod}}
\nc{\oamap}{\Phi^{\rm alg}}
\nc{\ommap}{\Phi^{\rm mod}}
\nc{\ioaset}{\mathbf{IO}^{\rm alg}}
\nc{\iomset}{\mathbf{IO}^{\rm mod}}
\nc{\ioamap}{\Psi^{\rm alg}}
\nc{\iommap}{\Psi^{\rm mod}}

\nc{\suc}{{successor}\xspace} \nc{\Suc}{{Successor}\xspace}
\nc{\sucs}{{successors}\xspace} \nc{\Sucs}{{Successors}\xspace}
\nc{\Tsuc}{{T-successor}\xspace}
\nc{\Tsucs}{{T-successors}\xspace} \nc{\Lsuc}{{L-successor}\xspace}
\nc{\Lsucs}{{L-successors}\xspace} \nc{\Rsuc}{{R-successor}\xspace}
\nc{\Rsucs}{{R-successors}\xspace}

\nc{\bia}{{$\mathcal{P}$-bimodule ${\bf k}$-algebra}\xspace}
\nc{\bias}{{$\mathcal{P}$-bimodule ${\bf k}$-algebras}\xspace}

\nc{\rmi}{{\mathrm{I}}}
\nc{\rmii}{{\mathrm{II}}}
\nc{\rmiii}{{\mathrm{III}}}

\nc{\pll}{\beta}
\nc{\plc}{\epsilon}

\nc{\ass}{{\mathit{Ass}}}
\nc{\lie}{{\mathit{Lie}}}
\nc{\comm}{{\mathit{Comm}}}
\nc{\dend}{{\mathit{Dend}}}
\nc{\zinb}{{\mathit{Zinb}}}
\nc{\tdend}{{\mathit{TDend}}}
\nc{\prelie}{{\mathit{preLie}}}
\nc{\postlie}{{\mathit{PostLie}}}
\nc{\quado}{{\mathit{Quad}}}
\nc{\octo}{{\mathit{Octo}}}
\nc{\ldend}{{\mathit{ldend}}}
\nc{\lquad}{{\mathit{LQuad}}}

 \nc{\adec}{\check{;}} \nc{\aop}{\alpha}
\nc{\dftimes}{\widetilde{\otimes}} \nc{\dfl}{\succ} \nc{\dfr}{\prec}
\nc{\dfc}{\circ} \nc{\dfb}{\bullet} \nc{\dft}{\star}
\nc{\dfcf}{{\mathbf k}} \nc{\apr}{\ast} \nc{\spr}{\cdot}
\nc{\twopr}{\circ} \nc{\tspr}{\star} \nc{\sempr}{\ast}
\nc{\disp}[1]{\displaystyle{#1}}
\nc{\bin}[2]{ (_{\stackrel{\scs{#1}}{\scs{#2}}})}  
\nc{\binc}[2]{ \left (\!\! \begin{array}{c} \scs{#1}\\
    \scs{#2} \end{array}\!\! \right )}  
\nc{\bincc}[2]{  \left ( {\scs{#1} \atop
    \vspace{-.5cm}\scs{#2}} \right )}  
\nc{\sarray}[2]{\begin{array}{c}#1 \vspace{.1cm}\\ \hline
    \vspace{-.35cm} \\ #2 \end{array}}
\nc{\bs}{\bar{S}} \nc{\dcup}{\stackrel{\bullet}{\cup}}
\nc{\dbigcup}{\stackrel{\bullet}{\bigcup}} \nc{\etree}{\big |}
\nc{\la}{\longrightarrow} \nc{\fe}{\'{e}} \nc{\rar}{\rightarrow}
\nc{\dar}{\downarrow} \nc{\dap}[1]{\downarrow
\rlap{$\scriptstyle{#1}$}} \nc{\uap}[1]{\uparrow
\rlap{$\scriptstyle{#1}$}} \nc{\defeq}{\stackrel{\rm def}{=}}
\nc{\dis}[1]{\displaystyle{#1}} \nc{\dotcup}{\,
\displaystyle{\bigcup^\bullet}\ } \nc{\sdotcup}{\tiny{
\displaystyle{\bigcup^\bullet}\ }} \nc{\hcm}{\ \hat{,}\ }
\nc{\hcirc}{\hat{\circ}} \nc{\hts}{\hat{\shpr}}
\nc{\lts}{\stackrel{\leftarrow}{\shpr}}
\nc{\rts}{\stackrel{\rightarrow}{\shpr}} \nc{\lleft}{[}
\nc{\lright}{]} \nc{\uni}[1]{\tilde{#1}} \nc{\wor}[1]{\check{#1}}
\nc{\free}[1]{\bar{#1}} \nc{\den}[1]{\check{#1}} \nc{\lrpa}{\wr}
\nc{\curlyl}{\left \{ \begin{array}{c} {} \\ {} \end{array}
    \right .  \!\!\!\!\!\!\!}
\nc{\curlyr}{ \!\!\!\!\!\!\!
    \left . \begin{array}{c} {} \\ {} \end{array}
    \right \} }
\nc{\leaf}{\ell}       
\nc{\longmid}{\left | \begin{array}{c} {} \\ {} \end{array}
    \right . \!\!\!\!\!\!\!}
\nc{\ot}{\otimes} \nc{\sot}{{\scriptstyle{\ot}}}
\nc{\otm}{\overline{\ot}}
\nc{\ora}[1]{\stackrel{#1}{\rar}}
\nc{\ola}[1]{\stackrel{#1}{\la}}
\nc{\pltree}{\calt^\pl}
\nc{\epltree}{\calt^{\pl,\NC}}
\nc{\rbpltree}{\calt^r}
\nc{\scs}[1]{\scriptstyle{#1}} \nc{\mrm}[1]{{\rm #1}}
\nc{\dirlim}{\displaystyle{\lim_{\longrightarrow}}\,}
\nc{\invlim}{\displaystyle{\lim_{\longleftarrow}}\,}
\nc{\mvp}{\vspace{0.5cm}} \nc{\svp}{\vspace{2cm}}
\nc{\vp}{\vspace{8cm}} \nc{\proofbegin}{\noindent{\bf Proof: }}
\nc{\proofend}{$\blacksquare$ \vspace{0.5cm}}
\nc{\freerbpl}{{F^{\mathrm RBPL}}}
\nc{\sha}{{\mbox{\cyr X}}}  
\nc{\ncsha}{{\mbox{\cyr X}^{\mathrm NC}}} \nc{\ncshao}{{\mbox{\cyr
X}^{\mathrm NC,\,0}}}
\nc{\shpr}{\diamond}    
\nc{\shprm}{\overline{\diamond}}    
\nc{\shpro}{\diamond^0}    
\nc{\shprr}{\diamond^r}     
\nc{\shpra}{\overline{\diamond}^r}
\nc{\shpru}{\check{\diamond}} \nc{\catpr}{\diamond_l}
\nc{\rcatpr}{\diamond_r} \nc{\lapr}{\diamond_a}
\nc{\sqcupm}{\ot}
\nc{\lepr}{\diamond_e} \nc{\vep}{\varepsilon} \nc{\labs}{\mid\!}
\nc{\rabs}{\!\mid} \nc{\hsha}{\widehat{\sha}}
\nc{\lsha}{\stackrel{\leftarrow}{\sha}}
\nc{\rsha}{\stackrel{\rightarrow}{\sha}} \nc{\lc}{\lfloor}
\nc{\rc}{\rfloor}
\nc{\tpr}{\sqcup}
\nc{\nctpr}{\vee}
\nc{\plpr}{\star}
\nc{\rbplpr}{\bar{\plpr}}
\nc{\sqmon}[1]{\langle #1\rangle}
\nc{\forest}{\calf}
\nc{\altx}{\Lambda_X} \nc{\vecT}{\vec{T}} \nc{\onetree}{\bullet}
\nc{\Ao}{\check{A}}
\nc{\seta}{\underline{\Ao}}
\nc{\deltaa}{\overline{\delta}}
\nc{\trho}{\tilde{\rho}}

\nc{\rpr}{\circ}
\nc{\dpr}{{\tiny\diamond}}
\nc{\rprpm}{{\rpr}}

\nc{\mmbox}[1]{\mbox{\ #1\ }} \nc{\ann}{\mrm{ann}}
\nc{\Aut}{\mrm{Aut}} \nc{\can}{\mrm{can}}
\nc{\twoalg}{{two-sided algebra}\xspace}
\nc{\colim}{\mrm{colim}}
\nc{\Cont}{\mrm{Cont}} \nc{\rchar}{\mrm{char}}
\nc{\cok}{\mrm{coker}} \nc{\dtf}{{R-{\rm tf}}} \nc{\dtor}{{R-{\rm
tor}}}

\nc{\depth}{{\mrm d}}
\nc{\Div}{{\mrm Div}} \nc{\End}{\mrm{End}} \nc{\Ext}{\mrm{Ext}}
\nc{\Fil}{\mrm{Fil}} \nc{\Frob}{\mrm{Frob}} \nc{\Gal}{\mrm{Gal}}
\nc{\GL}{\mrm{GL}} \nc{\Hom}{\mrm{Hom}} \nc{\hsr}{\mrm{H}}
\nc{\hpol}{\mrm{HP}} \nc{\id}{\mrm{id}} \nc{\im}{\mrm{im}}
\nc{\incl}{\mrm{incl}} \nc{\length}{\mrm{length}}
\nc{\LR}{\mrm{LR}} \nc{\mchar}{\rm char} \nc{\NC}{\mrm{NC}}
\nc{\mpart}{\mrm{part}} \nc{\pl}{\mrm{PL}}
\nc{\ql}{{\QQ_\ell}} \nc{\qp}{{\QQ_p}}
\nc{\rank}{\mrm{rank}} \nc{\rba}{\rm{RBA }} \nc{\rbas}{\rm{RBAs }}
\nc{\rbpl}{\mrm{RBPL}}
\nc{\rbw}{\rm{RBW }} \nc{\rbws}{\rm{RBWs }} \nc{\rcot}{\mrm{cot}}
\nc{\rest}{\rm{controlled}\xspace}
\nc{\rdef}{\mrm{def}} \nc{\rdiv}{{\rm div}} \nc{\rtf}{{\rm tf}}
\nc{\rtor}{{\rm tor}} \nc{\res}{\mrm{res}} \nc{\SL}{\mrm{SL}}
\nc{\Spec}{\mrm{Spec}} \nc{\tor}{\mrm{tor}} \nc{\Tr}{\mrm{Tr}}
\nc{\mtr}{\mrm{sk}}

\nc{\ab}{\mathbf{Ab}} \nc{\Alg}{\mathbf{Alg}}
\nc{\Algo}{\mathbf{Alg}^0} \nc{\Bax}{\mathbf{Bax}}
\nc{\Baxo}{\mathbf{Bax}^0} \nc{\RB}{\mathbf{RB}}
\nc{\RBo}{\mathbf{RB}^0} \nc{\BRB}{\mathbf{RB}}
\nc{\Dend}{\mathbf{DD}} \nc{\bfk}{{\bf k}} \nc{\bfone}{{\bf 1}}
\nc{\base}[1]{{a_{#1}}} \nc{\detail}{\marginpar{\bf More detail}
    \noindent{\bf Need more detail!}
    \svp}
\nc{\Diff}{\mathbf{Diff}} \nc{\gap}{\marginpar{\bf
Incomplete}\noindent{\bf Incomplete!!}
    \svp}
\nc{\FMod}{\mathbf{FMod}} \nc{\mset}{\mathbf{MSet}}
\nc{\rb}{\mathrm{RB}} \nc{\Int}{\mathbf{Int}}
\nc{\Mon}{\mathbf{Mon}}
\nc{\remarks}{\noindent{\bf Remarks: }}
\nc{\OS}{\mathbf{OS}} 
\nc{\Rep}{\mathbf{Rep}}
\nc{\Rings}{\mathbf{Rings}} \nc{\Sets}{\mathbf{Sets}}
\nc{\DT}{\mathbf{DT}}

\nc{\BA}{{\mathbb A}} \nc{\CC}{{\mathbb C}} \nc{\DD}{{\mathbb D}}
\nc{\EE}{{\mathbb E}} \nc{\FF}{{\mathbb F}} \nc{\GG}{{\mathbb G}}
\nc{\HH}{{\mathbb H}} \nc{\LL}{{\mathbb L}} \nc{\NN}{{\mathbb N}}
\nc{\QQ}{{\mathbb Q}} \nc{\RR}{{\mathbb R}} \nc{\BS}{{\mathbb{S}}} \nc{\TT}{{\mathbb T}}
\nc{\VV}{{\mathbb V}} \nc{\ZZ}{{\mathbb Z}}

\nc{\calao}{{\mathcal A}} \nc{\cala}{{\mathcal A}}
\nc{\calc}{{\mathcal C}} \nc{\cald}{{\mathcal D}}
\nc{\cale}{{\mathcal E}} \nc{\calf}{{\mathcal F}}
\nc{\calfr}{{{\mathcal F}^{\,r}}} \nc{\calfo}{{\mathcal F}^0}
\nc{\calfro}{{\mathcal F}^{\,r,0}} \nc{\oF}{\overline{F}}
\nc{\calg}{{\mathcal G}} \nc{\calh}{{\mathcal H}}
\nc{\cali}{{\mathcal I}} \nc{\calj}{{\mathcal J}}
\nc{\call}{{\mathcal L}} \nc{\calm}{{\mathcal M}}
\nc{\caln}{{\mathcal N}} \nc{\calo}{{\mathcal O}}
\nc{\calp}{{\mathcal P}} \nc{\calq}{{\mathcal Q}} \nc{\calr}{{\mathcal R}}
\nc{\calt}{{\mathcal T}} \nc{\caltr}{{\mathcal T}^{\,r}}
\nc{\calu}{{\mathcal U}} \nc{\calv}{{\mathcal V}}
\nc{\calw}{{\mathcal W}} \nc{\calx}{{\mathcal X}}
\nc{\CA}{\mathcal{A}}

\nc{\fraka}{{\mathfrak a}} \nc{\frakB}{{\mathfrak B}}
\nc{\frakb}{{\mathfrak b}} \nc{\frakd}{{\mathfrak d}}
\nc{\oD}{\overline{D}}
\nc{\frakF}{{\mathfrak F}} \nc{\frakg}{{\mathfrak g}}
\nc{\frakm}{{\mathfrak m}} \nc{\frakM}{{\mathfrak M}}
\nc{\frakMo}{{\mathfrak M}^0} \nc{\frakp}{{\mathfrak p}}
\nc{\frakS}{{\mathfrak S}} \nc{\frakSo}{{\mathfrak S}^0}
\nc{\fraks}{{\mathfrak s}} \nc{\os}{\overline{\fraks}}
\nc{\frakT}{{\mathfrak T}}
\nc{\oT}{\overline{T}}
\nc{\frakX}{{\mathfrak X}} \nc{\frakXo}{{\mathfrak X}^0}
\nc{\frakx}{{\mathbf x}}
\nc{\frakTx}{\frakT}      
\nc{\frakTa}{\frakT^a}        
\nc{\frakTxo}{\frakTx^0}   
\nc{\caltao}{\calt^{a,0}}   
\nc{\ox}{\overline{\frakx}} \nc{\fraky}{{\mathfrak y}}
\nc{\frakz}{{\mathfrak z}} \nc{\oX}{\overline{X}}

\font\cyr=wncyr10

\nc{\redtext}[1]{\textcolor{red}{#1}}

\makeatletter
\g@addto@macro{\endabstract}{\@setabstract}
\newcommand{\authorfootnotes}{\renewcommand\thefootnote{\@fnsymbol\c@footnote}}%
\makeatother
\begin{document}
\begin{center}
  \LARGE
\textbf{Block type Lie algebras and their representations}

  \normalsize
  \authorfootnotes
Xiaomin Tang   \footnote{Corresponding author: {\it X. Tang. Email:} x.m.tang@163.com}\textsuperscript{1},
Shasha Zhao\textsuperscript{1} \par \bigskip

   \textsuperscript{1}Department of Mathematics, Heilongjiang University, Harbin, 150080, P. R. China   \par

\end{center}


\begin{abstract}
Block type Lie algebras have been studied by many authors in the latest twenty years.
In this paper, we will study a class of more general Block type Lie algebra $\mathcal{B}(p,q)$, which is a class of infinite-dimensional Lie algebra by using the generalized Balinskii-Novikov's construction method to Witt type Novikov algebra. Rather, Block type Lie algebra $\mathcal{B}(p,q)$ with basis
$\{L_{\alpha,i},c|\alpha,i\in\mathbb{Z},i\geq0\}$ over $\mathbb{C}$ and relations:
$$[L_{\alpha,i},L_{\beta,j}]=((i+q)(\beta+p)-(j+q)(\alpha+p))
L_{\alpha+\beta,i+j}+\delta_{\alpha+\beta,0}\delta_{i,0}\delta_{j,0}\frac{\alpha^{3}-\alpha}{12}c,[c,L_{\alpha,i}]=0,
$$
where the parameter $p,q$ are some complex numbers.
We study the representation theory for $\mathcal{B}(p,q)$.  We classify quasifinite irreducible highest weight $\mathcal{B}(p,q)$-module.
We also prove that any quasifinite irreducible module  of Block type Lie algebras $\mathcal{B}(p,q)$ is either a highest or lowest weight module, or else a uniformly bounded module. This paper can be considered as a generalization of the related literatures.

\vspace{2mm}

\noindent{\it Keywords:}
Novikov algebra, Block type Lie algebra; highest weight module; uniformly bounded module;
quasi-finite module

\noindent{\it 2010 Mathematics Subject Classification}: 17B10; 17B65; 17B68.

\end{abstract}

\setcounter{section}{0}
{\ }

 \baselineskip=20pt

\section{Introduction}

Since a class of infinite dimensional simple Lie algebras were introduced by Block \cite{vin}, generalizations of Lie
algebras of this type (usually referred to as Block type Lie algebras) have been studied by many authors in the latest twenty years(see \cite{[Cay],[Ger],Os1,hc,Bai,BM2,KCB,Xia2016}). So far, the representation theory for Block type Lie algebras
is far from being well developed, except for quasifinite representations of some particular Block type Lie algebras. This article first get a class of more extensive Block type Lie algebra by Witt type Novikov algebra (see Section \ref{novikov}) as following
$$
\mathcal{B}(p,q)={\rm Span_{\mathbb{C}}} \{L_{\alpha,i}, c |\alpha,i\in \mathbb{Z}, i\ge 0\},
$$
satisfying
\begin{eqnarray*}
\begin{bmatrix}L_{\alpha,i},L_{\beta,j}\end{bmatrix}
=((i+q)(\beta+p)-(j+q)(\alpha+p))L_{\alpha+\beta,i+j}+\delta_{\alpha+\beta,0}\delta_{i,0}\delta_{j,0}\frac{\alpha^{3}-\alpha}{12}c,\nonumber \\
\begin{bmatrix}c,L_{\alpha,i}\end{bmatrix}=0, \;\;\forall \alpha,\beta,i,j\in{\mathbb Z}.
\end{eqnarray*}
The quasifinite representations of $\mathcal{B}(p,q)$ for some special cases are studied by many scholars.
For example, the authors \cite{BM1} studied the quasifinite representations for $\mathcal{B}(0,1)$, the authors  \cite{KCB,[Ger]} studied
the quasifinite representations of $\mathcal{B}(0,q)$ with $q\neq 0$.  Inspired by this, we will study the more general situation for $\mathcal{B}(p,q)$. Note that the case $p=0$ is considered by the  above mentioned papers, we study only the case $p\neq 0$.

Throughout the paper, we use $\mathbb{C}, \mathbb{C}^*,\mathbb{ Q}, \mathbb{Q}^*, \mathbb{Z}$
to denote respectively the sets of complex mumbers, nonzero complex numbers, rational number, nonzero rational number, integers, and $p,q$ is always assumed to be fixed numbers in $C^*$.

In this paper, we firstly study the Novikov algebra and Block type Lie algebras in section \ref{novikov} , we get a class of more extensive Block type Lie algebra by Witt type Novikov algebra. Then, in Section \ref{exmaple} we discuss a class of Block type Lie algebras and give two examples of their quasifinite modules, we also give the classification of the quasifinite highest weight modules. Finally, in Section \ref{classi} the classification of qusifinite module for a class of Block type Lie algebras is studied. We first start our main content from the following section.

\section{Novikov algebra and Block type Lie algebras}\label{novikov}

Novikov algebras were introduced in connection with Poisson brackets of hydrodynamic type \cite{GD1979} and Hamiltonian operators in the formal variational calculus \cite{BN}. The Novikov algebras are a class of Lie-admissible algebras whose commutators are Lie algebras. Thus it is useful to relate the study of Novikov algebras to the theory of Lie algebras.  Recall that

\begin{definition}
A Novikov algebra $A$ is a vector space over $\mathbb{C}$ with a bilinear product $(a, b)\rightarrow ab$ satisfying (for any $a, b, c \in A$)
\begin{equation}\label{nov1}
(ab)c-a(bc) = (ba)c-b(ac),
\end{equation}
\begin{equation}\label{nov2}
(ab)c = (ac)b.
\end{equation}
\end{definition}

 If we define an algebra $A$ only using the condition (\ref{nov1}), then it is called the Left-symmetric algebra. It is a class of Lie-admissible algebras whose commutators  $[x,y]= xy-yx$ defines a Lie algebra $\frak g=\frak g(A)$, which is called the {\it sub-adjacent} Lie algebra of $A$, and
conversely $A$ is called {\it compatible} on $\frak g(A)$. Left-symmetric algebra can be traced back to the study of rooted tree algebras (\cite{Cay}), geometry and deformation theory. They appear in many fields in mathematics and physics under different names like pre-Lie algebras, Vinberg algebras and quasi-associative algebras.

Clearly, Novikov algebra is a calss of special Left-symmetric algebra.
Balinskii-Novikov \cite{BN}  give a way to get a class of infinite-dimensional Lie algebra by Novikov algebra, which is called Balinskii-Novikov's construction. It has caught the interest of many scholars \cite{bp2010,bp2011,bp2012}. Recently,  Balinskii-Novikov's construction has been extended to many algebraic structures similar to Lie algebra, such as Lie superalgebras \cite{peib2012,WCB2012}, Lie conformal algebras \cite{YF2015} and Vertex algebras \cite{BLP2015}, and so on. In particular, any Novikov
algebra corresponds to an infinite-dimensional ``Virasoro-type"  Lie algebra by certain affinization given by Baliskii and Novikov  in \cite{BN}.
We slightly generality their conclusions, or rather one has

\begin{theorem}\label{theo-nov}
Let $A$ be a vector space over $\mathbb{C}$ with a bilinear product $(a, b)\rightarrow ab$,
  and let
 \begin{equation}
 L(A) = A \otimes \mathbb{C}[t, t^{-1}].
 \end{equation}
For a fixed complex number $q$, define a bilinear product $[,]: L(A)\times L(A) \rightarrow L(A)$ given by
\begin{equation}
[a[m],b[n]]=(m+q)(ab)[m+n]-(n+q)(ba)[m+n]
\end{equation}
where $a[m]:=a\otimes t^{m+1}$, for any $a,b\in A$ and $m,n\in \mathbb{Z}$.
 Then $L(A)$ is a Lie algebra if and only if $A$ is a Novikov algebra.
\end{theorem}
\begin{proof}
For any $a,b,c\in A$ and $m,n,k\in \mathbb{Z}$, we let
\begin{eqnarray*}
\Theta_1=(n+q)(m+q)a(bc)-(n+q)(n+k+q)(bc)a+(k+q)(n+k+q)(cb)a-(k+q)(m+q)a(cb),\\
\Theta_2=(m+q)(k+q)c(ab)-(m+q)(m+n+q)(ab)c+(n+q)(m+n+q)(ba)c-(n+q)(k+q)c(ba),\\
\Theta_3=(k+q)(n+q)b(ca)-(k+q)(m+k+q)(ca)b+(m+q)(m+k+q)(ac)b-(m+q)(n+q)b(ac).
\end{eqnarray*}
It follows by a simple computation that
\begin{eqnarray*}
\begin{bmatrix} a[m],\begin{bmatrix}b[n],c[k]\end{bmatrix}\end{bmatrix}=\Theta_1 [k+m+n],\\
\begin{bmatrix} c[k],\begin{bmatrix}a[m],b[n]\end{bmatrix}\end{bmatrix}=\Theta_2 [k+m+n],\\
\begin{bmatrix} b[n],\begin{bmatrix}c[k],a[m]\end{bmatrix}\end{bmatrix}=\Theta_3 [k+m+n].
\end{eqnarray*}
Further, we see that
\begin{eqnarray}
&&\Theta_1+\Theta_2+\Theta_3  \nonumber \\
&=&(n+q)(m+q)(a(bc)-b(ac))+ (m+q)(k+q)(c(ab)-a(cb))+(n+q)(k+q)(b(ca)-c(ba))\nonumber \\ \label{aa1}
&&+(n+q)(m+n+q)(ba)c-(n+q)(n+k+q)(bc)a \nonumber\\
&&+(k+q)(n+k+q)(cb)a-(k+q)(m+k+q)(ca)b \\
&&+(m+q)(m+k+q)(ac)b-(m+q)(m+n+q)(ab)c. \nonumber
\end{eqnarray}
If $A$ is a Lie algebra, the Jacobin equation tells us that
\begin{equation}\label{jacobi}
\begin{bmatrix} a[m],\begin{bmatrix}b[n],c[k]\end{bmatrix}\end{bmatrix}+
\begin{bmatrix} c[k],\begin{bmatrix}a[m],b[n]\end{bmatrix}\end{bmatrix}+
\begin{bmatrix} b[n],\begin{bmatrix}c[k],a[m]\end{bmatrix}\end{bmatrix}=(\sum_{i=1}^3\Theta_i) [k+m+n]=0,
\end{equation}
and so that $\Theta_1+\Theta_2+\Theta_3=0$. Review the equation (\ref{aa1}) as a formal polynomial in $m,n,k$.
By observing the coefficients of terms $m^2$ and $mn$  we have $(ab)c=(ac)b$ and $(ab)c-a(bc)=(ba)c-b(ac)$ for all $a,b,c\in A$.
This implies equations (\ref{nov1}) and (\ref{nov2}) hold, i.e., $A$ is a Novikov algebra.

Conversely, if $A$ is a Novikov algebra, it is enough to verify that the Jacobin equation hold.
Note that $(ab)c-a(bc)=(ba)c-b(ac)$ and $(ab)c=(ac)b$ for all $a,b,c\in A$, we have by (\ref{aa1}) that
\begin{eqnarray*}
&&\Theta_1+\Theta_2+\Theta_3 \\
&=&(n+q)(m+q)((ab)c-(ba)c))+ (m+q)(k+q)((ca)b-(ac)b)+(n+q)(k+q)((bc)a-(cb)a)\\
&&+(n+q)(m-k)(ba)c+(k+q)(n-m)(cb)a+(m+q)(k-n)(ac)b\\
&=&((n+q)(m+q)-(m+q)(k+q)+(m+q)(k-n))(ab)c\\
&&+((n+q)(m-k)-(m+q)(n+q)+(n+q)(k+q))(ba)c\\
&&+((m+q)(k+q)-(n+q)(k+q)+(k+q)(n-m))(ca)b\\
&=&0.
\end{eqnarray*}
This tells us that (\ref{jacobi}) holds. The proof is completed.
\end{proof}

\begin{remark}
When $q=1$, the above way is called Balinskii-Novikov's construction.  The authors \cite{bp2010,bp2011,bp2012} study some infinite-dimensional Lie algebras realized by Balinskii-Novikov's construction of some finite-dimensional Novikov algebras.
\end{remark}

 Note that the authors \cite{bp2010,bp2011,bp2012} study some infinite-dimensional Lie algebras realized by Balinskii-Novikov's construction of some finite-dimensional Novikov algebras but not by infinite-dimensional Novikov algebras. Inspired by this, we consider this type of problem for some infinite-dimensional Novikov algebras, which is related to Witt algebra (is called Witt type Novikov algebra ).

Recall that the Witt algebra $W$ (of rank one) which is a infinite-dimensional complex Lie algebra with a
basis $\{x_\alpha | \alpha \in \mathbb{Z}\}$ satisfying
\begin{equation}
[x_{\alpha}, x_{\beta}]=(\beta-\alpha)x_{\alpha+\beta},\;\;\forall \alpha,\beta\in \mathbb Z.\label{eq:W}
\end{equation}

The complex Witt algebra was first defined by Cartan (1909), and its analogues over finite fields were studied by Witt in the 1930s.
The Virasoro algebra is a central extension of the  Witt algebra, which have many applications in mathematics and physics. Recall
\cite{tb,KBC} given the compatible Novikov algebra on Witt algebra as following.
\begin{equation}\label{nong}
x_{\alpha}x_{\beta}=(\beta+p)x_{\alpha+\beta}+\mu x_{\alpha+\beta+\theta}, \forall \alpha,\beta\in \mathbb Z,
\end{equation}
where $p\in \mathbb{C}$ and $\theta \in \mathbb{Z}$ are fixed number. Now let us by using Theorem \ref{theo-nov} to above compatible Novikov algebra on Witt algebra and thus get a Lie algebra as follows.

\begin{proposition}\label{pro1}
For any $\alpha,\mu\in \mathbb{C}$ and $\theta\in \mathbb{Z}$, there is an infinite-dimensional Lie algebra
induced from the Novikov algebra on Witt algebra defined by (\ref{nong}), which is isomorphic to
the Lie algebra
$\mathcal{B}(p,q,\mu,\theta)$ with a $\mathbb C$-basis
$\{L_{\alpha,i}|\alpha,i\in \mathbb{Z}\}$ satisfying
\begin{equation}
[L_{\alpha,i},L_{\beta,j}]=((i+q)(\beta+p)-(j+q)(\alpha+p))L_{\alpha+\beta,i+j}+(i-j)\mu
L_{\alpha+\beta+\theta,i+j},\;\;\forall \alpha,\beta,i,j\in{\mathbb Z}.
\end{equation}
\end{proposition}

\begin{proof}
The conclusion follows from Theorem \ref{theo-nov} by setting $L_{\alpha,i}=x_\alpha \otimes t^{i+1}$ for any $m,i\in \mathbb Z$.
\end{proof}

Since a class of infinite dimensional simple Lie algebras were introduced by Block \cite{vin}, generalizations of Lie
algebras of this type (usually referred to as Block type Lie algebras) have been studied by many authors in the latest twenty years (see \cite{[Cay],[Ger],Os1,hc,Bai,BM2,KCB}).  We will see that the Block type Lie algebras  have a close relationship to the Lie algebra $\mathcal{B}(p,q,\mu,\theta)$ of Proposition \ref{pro1}.

\begin{example}
Consider the Lie algebra  $\mathcal{B}(p,q,\mu,\theta)$ given by Proposition \ref{pro1} when $p=1,q=s-1, \mu=-s, \theta=1$ for some $s$,
and denote $R_{\alpha,i}=L_{\alpha-1,i}$, we get the Lie algebra $ \mathcal{B}(1,s-1,-s,1)={\rm Span_{\mathbb{C}}} \{R_{\alpha,i}|\alpha,i\in \mathbb{Z}\},
$
satisfying
\begin{equation}
[R_{\alpha,i},R_{\beta,j}]=s(j-i)R_{\alpha+\beta,i+j}+((i+s-1)\beta-(j+s-1)\alpha)R_{\alpha+\beta-1,i+j}
\end{equation}

which is the Block Lie algebra $B(s,\mathbb{Z})$ studied by \cite{hb,su2007}.
\end{example}

\begin{example}\label{exm26}
Consider the Lie algebra  $\mathcal{B}(p,q,\mu,\theta)$ given by Proposition \ref{pro1} when $\mu=\theta=0$, that is
$$
\mathcal{B}(p,q,0,0)={\rm Span_{\mathbb{C}}} \{L_{\alpha,i}|\alpha,i\in \mathbb{Z}\},
$$
satisfying
\begin{equation}
[L_{\alpha,i},L_{\beta,j}]=((i+q)(\beta+p)-(j+q)(\alpha+p))L_{\alpha+\beta,i+j}, \;\;\forall \alpha,\beta,i,j\in{\mathbb Z}.
\end{equation}
Note that some subalgebras of $\mathcal{B}(p,q,0,0)$ or their central extension are called the {\bf{Block type Lie algebra}}.
Many authors study theory on constructions and representations of some Block type Lie algebra,
We list some literatures ( but not comprehensive ) as \cite{[Cay],[Ger],Bai,hb,BM2,KCB,BM1,su2015}.
\end{example}

\section{a class of Block type Lie algebra and some its modules} \label{exmaple}

In this section, a class of Block type Lie algebra $\mathcal{B}(p,q)$, which is a subalgebras of  $\mathcal{B}(p,q,0,0)$ of Example \ref{exm26} with the universal central extension, will come on the stage. That is, the Block type Lie algebra
$$
\mathcal{B}(p,q)={\rm Span_{\mathbb{C}}} \{L_{\alpha,i}, c |\alpha,i\in \mathbb{Z}, i\ge 0\},
$$
satisfying
\begin{eqnarray}
\begin{bmatrix}L_{\alpha,i},L_{\beta,j}\end{bmatrix}
=((i+q)(\beta+p)-(j+q)(\alpha+p))L_{\alpha+\beta,i+j}+\delta_{\alpha+\beta,0}\delta_{i,0}\delta_{j,0}\frac{\alpha^{3}-\alpha}{12}c,\nonumber \\
\begin{bmatrix}c,L_{\alpha,i}\end{bmatrix}=0, \;\;\forall \alpha,\beta,i,j\in{\mathbb Z}.
\end{eqnarray}

About the Block type Lie algebra $\mathcal{B}(p,q)$, we would like point out some facts as follows.

\begin{itemize}

\item We can realize the Lie algebra $\mathcal{B}(p,q)$ in the space $\mathbb{C}[x,x^{-1}]\otimes t^{q}\mathbb{C}[t]\oplus\mathbb{C}c$, with the bracket
\begin{eqnarray*}
&&[x^{\alpha}f(t),x^{\beta}g(t)]\\
&=&x^{\alpha+\beta}t^{1-q}((\beta+p)f^{\prime}(t)g(t)-(\alpha+p)f(t)g^{\prime}(t))+\delta_{\alpha+\beta,0}
\frac{\alpha^{3}-\alpha}{12}{\rm Res}_{t}(t^{-2q-1}f(t)g(t))c
\end{eqnarray*}
where $\alpha, \beta\in Z, f(t), g(t)\in t^{q}\mathbb{C}[t]$, and the prime stands for the derivative $\frac{d}{dt}$.  The corresponding relationship is $L_{\alpha,i}=x^{\alpha}t^{q+i}$.

\item The Lie algebra $\mathcal{B}(p,q)$ has a natural $\mathbb{Z}-$gradation $\mathcal{B}(p,q)=\oplus_{\alpha\in \mathbb{Z}}\mathcal{B}(p,q)_{\alpha}$ with
$$
\mathcal{B}(p,q)_{\alpha}=Span\{L_{\alpha,i}|i\in Z,i\geq 0\}\oplus\delta_{\alpha,0}\mathbb{C}c.
$$

\item For $q\neq 0$, $\mathcal{B}(p,q)$ contains the subalgebra $Vir$ isomorphic to the well-known Virasoro algebra, where
\begin{equation}
Vir={\rm Span}\{L_{\alpha},k|\alpha\in Z\}, L_{\alpha}\doteq q^{-1}L_{\alpha,0}, k\doteq q^{-2}c
\end{equation}
\begin{equation}
[L_{\alpha},L_{\beta}]=(\beta-\alpha)L_{\alpha+\beta}+\frac{\alpha^{3}-\alpha}{12}\delta_{\alpha+\beta,0}k, [k,L_{\alpha}]=0
\end{equation}

\item For $p,p^\prime,q,q^\prime\in \mathbb{Q}_{+}^{*}$, if $(p,q)\neq (p^\prime, q^\prime)$, then $\mathcal{B}(p,q) \not\cong \mathcal{B}(p^\prime, q^\prime)$. The authors \cite{xia2014} show that
$\mathcal{B}(0,q) \not\cong \mathcal{B}(0, q^\prime)$ while $q\neq q^\prime$. Note that the status of $p$ and $q$ in $\mathcal{B}(p,q)$ are equal, it is easy to obtain the same assertion in the similar way.

\item The quasifinite representations of $\mathcal{B}(p,q)$ for some special cases are studied by many scholars.
For example, the authors \cite{BM1} studied the quasifinite representations for $\mathcal{B}(0,1)$, the authors  \cite{KCB,[Ger]} studied
the quasifinite representations of $\mathcal{B}(0,q)$.
\end{itemize}

 Inspired by this, we will study the more general situation for $\mathcal{B}(p,q)$ in the rest of this paper.
Now Let us recall some important definition about the representation (or module) as follows:

\begin{definition}
A module $V$ over $\mathcal{B}(p,q)$ is called

\begin{itemize}

\item $\mathbb{Z}-$graded if $V=\oplus_{\alpha\in \mathbb{Z}} V_{\alpha}$ and $\mathcal{B}(p,q)_{\alpha}V_{\beta}\subset V_{\alpha+\beta}$ for all $\alpha,\beta$;

\item quasifinite if it is $\mathbb{Z}-$graded and $\dim V_{\beta}<\infty$ for all $\beta$;

\item uniformly bounded if it is $\mathbb{Z}-$graded and there is $N\geq 0$ with $dimV_{\beta}\leq N$ for all $\beta$;

\item a module of the intermediate series if it is $\mathbb{Z}-$graded and $dimV_{\beta}\leq 1$ for all $\beta$;

\item a highest(resp.lowest) weight module if there exists some $\wedge\in\mathcal{B}(p,q)_{0}^{*}$ such that $V=V(\wedge)$,
 where $V(\wedge)$ is a module generated by a highest (resp.lowest) weight vector $v_{\wedge}\in V(\wedge)_{0}$, $v_{\wedge}$ satisfies $h. v_{\wedge}=\wedge(h)v_{\wedge}, h\in\mathcal{B}(p,q)_{0}$, $\mathcal{B}(p,q)_{+}v_{\wedge}=0$ (resp. $\mathcal{B}(p,q)_{-}v_{\wedge}=0$),
 where $\mathcal{B}(p,q)_{\pm}=\oplus_{\pm\alpha>0}\mathcal{B}(p,q)_{\alpha}$.

 \end{itemize}
\end{definition}

Using the gradation, we introduce the following notations:
$$\mathcal{B}(p,q)_{[\beta,\gamma]}=\sum_{\beta\leq\alpha\leq\gamma}\mathcal{B}(p,q)_{\alpha}~~\beta,\gamma\in\mathbb{Z}.$$
similarly for $\mathcal{B}(p,q)_{[\beta,+\infty]},\mathcal{B}(p,q)_{[\beta,\gamma)}$ and so on. We have the following triangular decomposition: \begin{equation}
\mathcal{B}(p,q)=\mathcal{B}(p,q)_{-}\oplus\mathcal{B}(p,q)_{0}\oplus\mathcal{B}(p,q)_{+}
\end{equation}

When we study representations  of a Lie algebra of this kind , we encounter the difficulty that though it is $\mathbb{Z}-$
graded, the graded subspaces are still infinite dimensional, thus the study of quasifinite modules is a nontrivial probiem. We
give two examples of quasifinite modules as follows.

\subsection{The modules of the intermediate series}

Note that $\mathcal{B}(p,q)$ contains a subalgebra isomorphic to Virasoro algebra. Similar to the case of Virasoro algebra, we now give the irreducible $\mathcal{B}(p,q)$-modules of the intermediate series, i.e.,  $A_{a,b},A_{a},B_{a},a,b\in \mathbb{C}$, with the same basics $\{v_{\mu}|\mu\in\mathbb{Z}\}$ of $A_{a,b}, A_{a}$ and  $B_{a}$ and $\mathcal{B}(p,q)$ acts on their respectively as follows.
\begin{itemize}
\item $A_{a,b}:L_{\alpha,0}v_{\mu}=q(a+\mu+b\alpha)v_{\alpha+\mu},\ \ L_{\alpha,i}v_{\mu}=0\ (i>0),$

\item $A_{\alpha}:L_{\alpha,0}v_{\mu}=q(\mu+\alpha)v_{\alpha+\mu}(\mu\neq0),\ L_{\alpha,0}v_{0}=q\alpha(a+\alpha)v_{\alpha},
L_{\alpha,i}v_{\mu}=0\ (i>0),$

\item $
B_{a}:L_{\alpha,0}v_{\mu}=q\mu v_{\alpha+\mu}(\mu\neq-\alpha),\ L_{\alpha,0}v_{-\alpha}=-q\alpha(a+\alpha)v_{0},
L_{\alpha,i}v_{\mu}=0\ (i>0).$
\end{itemize}
It is obvious that the modules of the intermediate series are a class of module of uniformly bounded.

\subsection{Quasifinite highest weight modules }

Now we want to study the quasifinite highest weight modules of $\mathcal{B}(p,q)$, the lowest weight modules are similar.
To sate our main result in this part, we need to introduce the generating series and Verma module.

\begin{itemize}
\item
For any function $\Lambda\in \mathcal{B}(p,q)_0^*$, we set labels $\Lambda_{i}=\Lambda(L_{0,i})$ for $i\ge 0$,
and define $\Delta_{\Lambda}(z,p,q)$ as a following generating series with variable $z$,
\begin{equation}
\Delta_{\Lambda}(z,p,q)=2q\sum_{i=0}^{\infty}\frac{z^{i}}{i!}\Lambda_{i}
+\sum_{i=0}^{\infty}(1-p^{2})\frac{z^{i+1}}{i!}\Lambda_{i+1}=\Lambda((2q+(1-p^{2})zt)t^{q}e^{zt}).
\end{equation}
Recall that a quasipolynomial is a linear combination of functions of the form $f(z)e^{bz}$ , where $f(z)\in\mathbb{C}[z]$,$b\in\mathbb{C}$.

\item A $Verma$ module over $\mathcal{B}(p,q)$ is defined as the induced module
$$M(\Lambda)=U(\mathcal{B}(p,q))\otimes_{U(\mathcal{B}(p,q)_{0}\oplus\mathcal{B}(p,q)_{+}})\mathbb{C}v_{\Lambda},\Lambda\in\mathcal{B}(p,q)_{0}^{*},$$
where $\mathbb{C}v_{\Lambda}$ is the one dimensional $\mathcal{B}(p,q)_{0}\oplus\mathcal{B}(p,q)_{+}$-module and satisfies  $(h+n)(v_{\Lambda})=\Lambda(h)v_{\Lambda}$, here $h\in\mathcal{B}(p,q)_{0}$, $n\in\mathcal{B}(p,q)_{+}$ and $U(\mathcal{B}(p,q))$ stands for the universal enveloping algebra of $\mathcal{B}(p,q)$. Then any highest weight module $V(\Lambda)$ is a quotient module of $M(\Lambda)$ and the irreducible highest weight module $L(\Lambda)$ is the quotient of $M(\Lambda)$ by the maximal proper $\mathbb{Z}$-graded submodule.

\end{itemize}

Our main result in this part is following.

\begin{theorem}\label{theo123}
Let $L(\Lambda)$ be an irreducible highest weight module over $\mathcal{B}(p,q)$ with highest wight $\Lambda\in\mathcal{B}(p,q)_{0}^{*}$.
 Then  $L(\Lambda)$ is quasifinite if and only if $\Delta_{\Lambda}(z,p,q)$ is a quasi-polynomial.
\end{theorem}

In order to prove the result, we give some definitions and preliminary results.

\begin{definition}\label{def11}
Let $\mathcal{L}=\oplus_{\alpha\in\mathbb{Z}}\mathcal{L}_{\alpha}$ be a $\mathbb{Z}$-graded Lie algebra.

\begin{enumerate}
 \item \label{33a}
 A subalgebra $P$ of $\mathcal{L}$ is called parabolic if it contains $\mathcal{L}_{0}+\mathcal{L}_{+}$
 as a proper subalgebra,namely, $P=\oplus_{\alpha\in\mathbb{Z}}P_{\alpha}$ with $P_{\alpha}=\mathcal{L}_{\alpha},\alpha\geq0$, for some $\alpha<0$ such that $P_{\alpha}\neq0$.

\item \label{33b}
Given $0\neq a\in \mathcal{L}_{-1}$, we define a parabolic subalgebra
$\mathcal{P}(a)=\oplus_{\alpha\in Z}\mathcal{P}(a)_{\alpha}$ of $\mathcal{L}$ as follows:
\begin{equation*}
\mathcal{P}(a)_{\alpha}=
\left\{ \begin{array}{ll}
\mathcal{L}_{\alpha} & { \mbox if} \ \alpha\geq0,\\
{\rm Span} \{[\cdots,[\mathcal{L}_{0},[\mathcal{L}_{0},a]]\cdots]\} & { \mbox if} \ \alpha=-1,\\
 \begin{bmatrix}\mathcal{P}(a)_{-1},\mathcal{P}(a)_{\alpha+1}\end{bmatrix}& { \mbox if} \ \alpha\leq-2.
\end{array}
\right.
\end{equation*}

\item \label{33c}
A parabolic subalgebra $\mathcal{P}$ is called nondegenerate if $\mathcal{P}_{\alpha}$ has finite codimension in $\mathcal{L}_{\alpha}$, for all $\alpha<0,$.

\item \label{33d}
A nonzero element $a$ in $\mathcal{L}_{-1}$ is called nondegenerate if $\mathcal{P}(a)$ is  nondegenerate.

\end{enumerate}
\end{definition}

Define a parabolic subalgebra $\mathcal{B}(p,q,a)=\oplus_{\alpha\in\mathbb{Z}}\mathcal{B}(p,q,a)_{\alpha}$ of
$\mathcal{B}(p,q)$ as in Definition \ref{def11} (\ref{33b}), where $a\in\mathcal{B}(p,q)_{-1}$.
 By \cite{BM1,BM2,Bai}, $\mathcal{B}(p,q,a)$ is the minimal parabolic subalgebra containing $a$ and
$$\mathcal{B}(p,q,a)_{0}\doteq\mathcal{B}(p,q)_{0}\cap[\mathcal{B}(p,q,a),\mathcal{B}(p,q,a)]=[a,\mathcal{B}(p,q)_{1}].$$
It is clear that $$[x^{-1}f(t),xg(t)]=t^{1-q}((f(t)g(t))^{\prime}+p(f^{\prime}(t)g(t)-f(t)g^{\prime}(t)))$$ and
$$
x^{-1}f(t)\in\mathcal{B}(p,q)_{-1},xg(t)\in\mathcal{B}(p,q)_{1},
$$
this implies
\begin{equation}\label{16}
\mathcal{B}(p,q,a)_{0}={\rm Span} \{ t^{1-q}((f(t)g(t))^{\prime}+p(f^{\prime}(t)g(t)-f(t)g^{\prime}(t)))|g(t)\in t^{q}\mathbb{C}[t]\}.
\end{equation}

\begin{lemma}
Let $\mathcal{P}(p,q)=\oplus_{\alpha\in\mathbb{Z}}\mathcal{P}(p,q)_{\alpha}$ be the parabolic subalgebra of
$\mathcal{B}(p,q)$.

\begin{enumerate}

\item \label{34a}
 There exists $0 \neq a\in\mathcal{B}(p,q)_{-1}$ such that $\mathcal{P}(p,q,a)\subseteq\mathcal{P}(p,q)$;

\item \label{34b}
For any $\alpha<0$, the subspace $\mathcal{P}(p,q)_{\alpha}$ is nontrivial, and has finite codimension in
$\mathcal{B}(p,q)_{\alpha}$;

\item \label{34c}
$\mathcal{P}(p,q)$ is nondegenerate , and for any $0\neq a\in\mathcal{B}(p,q)_{-1}$ is nondegenerate.

\end{enumerate}
\end{lemma}

\begin{proof}
By definition, there exists at least one $\alpha<0$, such that $\mathcal{P}(p,q)_{\alpha}\neq0$.
 We claim that $\mathcal{P}(p,q)_{\alpha+1}\neq\{0\}$ if $\alpha\leq-2$.
 Otherwise, $[\mathcal{P}(p,q)_{\alpha},\mathcal{B}(p,q)_{1}]\subseteq\mathcal{P}(p,q)_{\alpha+1}=\{0\}$ and so that $[\mathcal{P}(p,q)_{\alpha},\mathcal{B}(p,q)_{1}]=\{0\}$.
Since $\alpha<0$, we can easily choose some positive integer $j_{0}$ such that  $k_{\alpha}=p(i-j_{0})-\alpha(q+j_{0})+q+i\neq0$
for all $i\in\mathbb{Z}_{+}$. Taking any $0\neq b=\sum_{i\in\mathbb{Z}}b_{i}L_{\alpha,i}$
in $\mathcal{P}(p,q)_{\alpha}$,   where $I$ is a finite subset of $\mathbb{Z}_{+}$ and $b_{i}\in\mathbb{C}$. We have
\begin{equation}
0=[b,L_{1,j_{0}}]=[\sum_{i\in I}b_{i}L_{\alpha,i},L_{1,j_{0}}]
=\sum_{i\in I}b_{i}[L_{\alpha,i},L_{1,j_{0}}]=\sum_{i\in I}b_{i}k_{\alpha}L_{\alpha+1,1+j_{0}}.
\end{equation}
Thus, we have $b_{i}=0$  for all $i\in I$. In the other word, $b=0$, this contradicts $b\neq0$. This proves the claim.
Therefore  $\mathcal{P}(p,q)_{-1}\neq\{0\}$ by induction. Taking any nonzero element $a\in\mathcal{P}(p,q)_{-1}$, note that
$\mathcal{P}(p,q,a)$ is the minimal  parabolic subalgebra, so we have $\mathcal{P}(p,q,a)\subseteq\mathcal{P}(p,q)$. The proof of (\ref{34a}) is completed.
One can prove parts (\ref{34b}) and (\ref{34c}) in a similar way as in \cite{KCB}, and the details are omitted.
\end{proof}

Suppose $\Lambda\in\mathcal{B}(p,q)_{0}^{*}$ satisfies $\Lambda|_{\mathcal{B}(p,q)_{0}\cap[\mathcal{P}(p,q),\mathcal{P}(p,q)]}=0$, we have for all $\alpha<0$, the action of $\mathcal{P}(p,q)_{\alpha}$ on $v_{\Lambda}$ is zero. Then we can define a highest weigh $\mathcal{B}(p,q)$-module as follows:
$$M(\mathcal{P}(p,q),\Lambda)=U(\mathcal{B}(p,q))\otimes U(\mathcal{P}(p,q))\mathbb{C}v_{\Lambda}.$$
Then $M(\mathcal{P}(p,q),\Lambda)$ is called the generalized $Verma$ module of $\mathcal{B}(p,q)$.

\begin{lemma}\label{lemm11}
The following conditions on $\Lambda\in\mathcal{B}(p,q)_{0}^{*}$ are equivalent.

\begin{enumerate}
\item \label{35a}
$L(\Lambda)$ is quasifinite;

\item \label{35b}
There exists an element $0\neq a\in\mathcal{B}(p,q)_{-1}$, such that $\Lambda(\mathcal{B}(p,q,a)_{0})=0$;

\item \label{35c}
$M(\Lambda)$ contains a singular vector $a\cdot v_{\Lambda}\in M(\Lambda)_{-1}$, where $0\neq a\in\mathcal{B}(p,q)_{-1}$;

\item \label{35d}
There exists an element $0\neq a\in\mathcal{B}(p,q)_{-1}$, such that $L(\Lambda)$ is
an irreducible quotient of the generalized $Verma$-module $M(\mathcal{P}(p,q,a),\Lambda)$.
\end{enumerate}
\end{lemma}
\begin{proof}
One can prove this lemma in a similar way as in Lemma 3.1 and theorem 2.5 of \cite{Os1}, and the details are omitted.
\end{proof}

{\bf The proof of Theorem \ref{theo123}:}
\begin{proof}
We use notation $e^{zt}=\sum_{i=0}^{\infty}\frac{z^{i}}{i!}t^{i}$ as generating series of $\mathbb{C}[t]$.
For all $f(t)\in \mathbb{C}[t]$, we have $f(t)e^{zt}=f(\frac{\partial}{\partial z})e^{zt}$.
For all $f(t)\in t^{q}\mathbb{C}[t]$, letting $\widetilde{f}(t)\doteq t^{-q}f(t)\in\mathbb{C}[t]$.
Thus $f(t)=\widetilde{f}(\frac{\partial}{\partial z})(t^{q}e^{zt})$.
Therefore, due to Lemma \ref{lemm11}(\ref{35b}),(\ref{35c}), when $L(\Lambda)$ is quasifinite we can find $0\neq f(t)\in t^{q}\mathbb{C}[t]$ with
$g(t)\in t^{q}\mathbb{C}[t]$ such that
\begin{equation}\label{17}
\Lambda(((f(t)g(t))^{\prime}+p\cdot g^{2}(t)\cdot(\frac{f(t)}{g(t)})^{\prime})t^{1-q})=0,
\end{equation}
where the prime stands for the partial derivative $\frac{\partial}{\partial t}$. Taking $g(t)=t^{q}e^{zt}$, we deduce
\begin{eqnarray}\label{18}
&&\Lambda(((f(t)g(t))^{\prime}+p\cdot g^{2}(t)\cdot(\frac{f(t)}{g(t)})^{\prime})t^{1-q}) \nonumber\\
&= &\Lambda((((f(t)t^{q}e^{zt})^{\prime}+p\cdot(t^{q}e^{zt})^{2}(\frac{f(t)e^{zt}}{t^{q}e^{2zt}})^{\prime})t^{1-q})\nonumber\\
&=&\Lambda(\widetilde{f}(\frac{\partial}{\partial z})((2q+zt)t^{q}e^{zt}-p^{2}t^{q}zt e^{zt}))\\
&=&\widetilde{f}(\frac{d}{dz})\Lambda((2q+(1-p^{2})zt)t^{q}e^{zt})=\widetilde{f}(\frac{d}{dz})\Delta_{\Lambda}(z,p,q).\nonumber
\end{eqnarray}
So we have by (\ref{17}) and (\ref{18}) that $\widetilde{f}(\frac{d}{dz})\Delta_{\Lambda}(z,p,q)=0$.
A well known fact \cite{[Cay]} stated that a formal power series is a quasipolynomial if
and only if it satisfies a nontrivial linear differential equation with constant coefficient. It is easy to see that
$\Delta_{\Lambda}(z,p,q)$ is a quasipolynomial.

Conversely, if $\Delta_{\Lambda}(z,p,q)$ is quasipolynomial, there exists $0\neq h(t)\in \mathbb{C}[t]$, such that
$h(\frac{d}{dz})\Delta_{\Lambda}(z,p,q)=0$. Denote $f(t)=t^{q}h(t)\in t^{q}\mathbb{C}[t]$,
then $\widetilde{f}(\frac{d}{dz})\Delta_{\Lambda}(z,p,q)=0$. It is follows by (\ref{18}) that
\begin{eqnarray*}
0&=&\Lambda(((f(t)g(t))^{\prime}+p(f^{\prime}(t)g(t)-g^{\prime}(t)f(t)))t^{1-q})\\
&=&\sum_{i=0}^{\infty}\frac{z^{i}}{i!}\Lambda(((f(t)t^{q+i})^{\prime}+p(f^{\prime}(t)t^{q+i}-(t^{q+i})^{\prime}f(t)))t^{1-q}),
\end{eqnarray*}
which yields that$\Lambda(((f(t)t^{q+i})^{\prime}+p(f^{\prime}(t)t^{q+i}-(t^{q+i})^{\prime}f(t)))t^{1-q})=0$  for all $i\in\mathbb{Z}_{+}$.
Hence, we have $\Lambda(((f(t)g(t))^{\prime}+p(f^{\prime}(t)g(t)-g^{\prime}(t)f(t)))t^{1-q})=0$ for all $g(t)\in t^{q}\mathbb{C}[t]$.
Due to Lemma \ref{lemm11}(\ref{35b}) and (\ref{16}), we deduce that $L(\Lambda)$ is quasifinite. The proof is completed.
\end{proof}

Here there is a problem in nature, is there the other quasifinite irreducible module of $\mathcal{B}(p,q)$ except uniformly bounded or
a highest (resp.lowest) weight of module? We will answer it in the next section.

\section{The classification of qusifinite module}\label{classi}

Motivated by a well-known result of Mathieu's in \cite{Mathieu}, it is natural to consider
the classification of quasifinite irreducible $\mathcal{B}(p,q)$-module, we shall prove the  main results as follows:
\begin{theorem}\label{theo}
A quasifinite irreducible $\mathcal{B}(p,q)$-module is either a highest/lowest weight module, or a uniformly bounded
module.
\end{theorem}

Note that $\mathcal{B}(p,q)_{0}={\rm Span} \{L_{0,i}|i\in \mathbb{Z},i\geq 0\}\oplus\mathbb{C}c$ is an infinite dimensional commutative subalgebra of $\mathcal{B}(p,q)$
(but not a Cartan subalgebra). Suppose $V=\oplus_{\mu\in \mathbb{Z}}V_{\mu}$ is a quasifinite $\mathcal{B}(p,q)$-module. Taking $\mu_{0}\in Z^{*}$,
Then $c|_{V_{\mu_0}}$ (the action of $c$ on $V_{\mu_{0}}$) and $L_{0,i}|_{{V_{\mu_{0}}}},i\in Z_{+}$  are linear transformations of the finite
dimensional subspace $V_{\mu_{0}}$. So there exists big enough fixed interer $p_{0}$ such that$c|_{v_{\mu_0}}, L_{0,0}|_{v_{\mu_0}}, \cdots, L_{o,m-1}|_{V_{\mu_{0}}}$ are linear dependent for all $m\geq p_{0}$. Therefore, for any $m\geq p_{0}$, we have
$a_{0}c|_{V_{\mu_{0}}}+a_{1}L_{0,0}|_{V_{\mu_{0}}}+\cdots +a_{m+1}L_{0,m-1}|_{V_{\mu_{0}}}=0$ for some not all zero complex numbers $a_0, a_1, \cdots, a_{m-1}$. That is, for all $v\in V_{\mu_{0}}$,
\begin{equation}\label{100}
(a_{0}c|_{V_{\mu_{0}}}+a_{1}L_{0,0}|_{V_{\mu_{0}}}+\cdots +a_{m-1}L_{0,m-1}|_{V_{\mu_{0}}})v=0.
\end{equation}

Now we define the Lie subalgebra $L(p,q,\mu_{o})$ of $\mathcal{B}(p,q)$ as follows
$$
L(p,q,\mu_{o})=
\left\{ \begin{array}{ll}
\langle L_{-\mu_{0},0},L_{-\mu_{0},1},L_{-\mu_{0},2},L_{-\mu_{0}+1,0},L_{0,m-1}|m\geq p_{0}\rangle, & \ {\mbox if}\ \mu_{0}\leq-1,\\
\\
\\
\langle L_{-\mu_{0},0},L_{-\mu_{0},1},L_{-\mu_{0},2},L_{-\mu_{0}-1,0},L_{0,m-1}|m\geq p_{0}\rangle, & \ {\mbox if}\ \mu_{0}\geq ~1,
\end{array}
\right.
$$
where the angle bracket $\langle, \rangle$ stands for ``the Lie subalgebra generated by".

\begin{lemma}\label{lemm21}
For any $n\geq 1$, and fixed $\mu_{0}\in Z^{*}$, we have

\begin{enumerate}
\item \label{42a}
if $\mu_{0}\leq-1$,
then there exists $\alpha_{n}\in\mathbb{Z_{+}^{*}}$ such that
$L_{\alpha,n-1}\in L(p,q,\mu_{0})$ for all $\alpha\geq\alpha_{n};$

\item \label{42b}
 if  $\mu_{0}\geq ~1$, then there exists $\alpha_{n}\in\mathbb{Z_{-}^{*}}$ such that $L_{\alpha,n-1}\in L(p,q,\mu_{0})$
for all $\alpha\leq\alpha_{n}$.
\end{enumerate}
\end{lemma}
\begin{proof}
We only prove part (\ref{42a}) by induction on $n$ (part (\ref{42b}) can be proved similarly).
In case $n=1$, using the same way of Lemma 2.1 in  \cite{KCB}, we have for any integer $\alpha\geq(1-\mu_{0})^{2}$, there exists two positive integer $l_{1},l_{2}$
such that
\begin{equation}
\alpha=l_{1}(1-\mu_{0})-l_{2}\mu_{0},
\end{equation}
where $l_{1}\doteq\alpha+(k_{0}+1)\mu_{0}\geq k_{0}(1-\mu_{0})+(k_{0}+1)\mu_{0}$, $l_{2}\doteq(k_{0}+1)(1-\mu_{0})-\alpha>0$ and
$k_{0}\doteq[\frac{\alpha}{1-\mu_{0}}]$.

 Letting $z_{1}=L_{-\mu_{0}+1,0},z_{2}=L_{-\mu_{0},0}$, taking
$\alpha_{1}=(1-\mu_{0})^{2}$, by  induction on $l_{1},l_{2}$,
we get \begin{equation}\label{26}
ad_{z_{2}}^{l_{2}-1}ad_{z_{1}}^{l_{1}}(z_{2})=q^{l_{1}+l_{2}-1}\prod_{i=1}^{l_{1}}(-(i-1)\mu_{0}+i-2)
\prod_{j=1}^{l_{2}-1}(-(l_{1}+j-1)\mu_{0}+l_{1})L_{\alpha,0}.
\end{equation}
Note that the coefficient of $L_{\alpha,0}$ on the right-hand side of (\ref{26}) is nonzero.
Thus $L_{\alpha,0}\in L(p,q,\mu_{0})$. Now suppose $n=s-1$, there exists an integer $\alpha_{s-1}\in Z_{+}^{*}$, such that $L_{\alpha,s-2}\in L(p,q,\mu_{0})$
for all $\alpha\geq\alpha_{s-1}$. Thus, when $n=s$, letting $r_{\alpha,p,q}=\alpha(q+2)+\mu_{0}(s+2q-1)-p(s-5)$.
If $s=3,q=-1$ then $r_{\alpha,p,q}=\alpha-2p$, otherwise, letting $r_{\alpha,p,q}=\alpha(q+1)+\mu_{0}(s+2q-1)-p(s-3)$.
We can always choose big enough $\alpha_{s}^{\prime}$ such that $r_{\alpha,p,q}\neq0$ whenever $\alpha\geq\alpha_{s}^{\prime}$.
Now take $\alpha_{s}=max\{\alpha_{s-1},\alpha_{s}^{\prime}\}$, then for all $\alpha\geq\alpha_{s}$, there exists an integer $\alpha_{s}\in Z_{+}^{*}$, such that
$$
L_{\alpha,s-1}=
\left\{ \begin{array}{ll}
L_{\alpha,s-1}=-\frac{1}{r_{\alpha,p,q}}[L_{\alpha+\mu_{0},s-3},L_{-\mu_{0},2}], & \ {\mbox if}\ s=3,q=-1,\\
\\
\\
L_{\alpha,s-1}=-\frac{1}{r_{\alpha,p,q}}[L_{\alpha+\mu_{0},s-2},L_{-\mu_{0},1}], & \ {\mbox else},
\end{array}
\right.
$$
which shows $L_{\alpha,n-1}\in L(p,q,\mu_{0})$. Part (\ref{42a}) is proved.
\end{proof}

One can prove a lemma in a similar way as in \cite{BM1,BM2,Bai} as follows.

\begin{lemma}\label{lem23}
Let $V=\oplus_{\mu\in\mathbb{Z}}V_{\mu}$ be a quasifinite  irreducible $\mathcal{B}(p,q)$-module.

\begin{enumerate}
\item \label{43a}
if $\mu_{0}\leq-2$, and there exists $0\neq v_{0}\in V_{\mu_{0}}$ such that $\mathcal{B}(p,q)_{[\alpha,+\infty)}v_{0}=0$ for some $\alpha>0$,
 then $V$ has a highest weight vector.

\item \label{43b}
 if $\mu_{0}\geq~2$, and there exists $0\neq v_{0}\in V_{\mu_{0}}$ such that $\mathcal{B}(p,q)_{(-\infty,\alpha]}v_{0}=0$ for some $\alpha<0$,
 then $V$ has a lowest  weight vector.
 \end{enumerate}
\end{lemma}

{\bf The proof of Theorem \ref{theo}:}
\begin{proof}
Assume that $V=\oplus_{\mu\in\mathbb{Z}}V_{\mu}$ is a quasifinite irreducible $\mathcal{B}(p,q)$-module without
highest and lowest weight vectors. We should prove that
\begin{equation}
\dim V_{\mu}\leq
\left\{ \begin{array}{ll}
3\dim V_{0}+\dim V_{1}, & \ {\mbox if}\ \mu\leq-2,\\
\\
\\
3 \dim V_{0}+\dim V_{-1}, & \ {\mbox if}\ \mu\geq 2.
\end{array}
\right.
\end{equation}
For fixed $\mu_{0}\leq-2$, we claim that the following linear map is injective:
$$\theta_{\mu_{0}}=(L_{-\mu_{0},0}\oplus L_{-\mu_{0},1}\oplus L_{-\mu_{0},2}\oplus L_{-\mu_{0}+1,0})|_{V_{\mu_{0}}}:V_{\mu_{0}}\rightarrow V_{0}\oplus V_{0}\oplus V_{0}\oplus V_{1}.$$
Otherwise there exists $0\neq v_{0}\in V_{\mu_{0}}$, such that $\theta_{\mu_{0}}(v_{0})=0$,
which implies that $L_{-\mu_{0},0},L_{-\mu_{0},1}L_{-\mu_{0},2},L_{-\mu_{0}+1,0}$
take $v_{0}$ to zero. On the other hand, $(a_{0}c|_{V_{\mu_{0}}}+a_{1}L_{0,0}|_{V_{\mu_{0}}}+\cdots +a_{m-1}L_{0,m-1}|_{V_{\mu_{0}}})v_{\mu_{0}}=0$ for
$m\geq p_{0}$ by (\ref{100}).  Hence, by definition of $L(p,q,\mu_{o})$,
\begin{equation}\label{24}
L(p,q,\mu_{0})v_{\mu_{0}}=0.
\end{equation}
Applying Lemma \ref{lemm21}(\ref{42a}), for any $1\leq m\leq p_{0}$, there exists some positive integer $\alpha_{m}$
such that $L_{\alpha,m-1}\in L(p,q,\mu_{0})$ for $\alpha\geq\alpha_{m}$.
Denote $\Gamma=\max \{\alpha_{1},\alpha_{2}\cdots\cdots\alpha_{p_{0}-1}\}$.
Then $L_{\alpha,m-1}\in L(p,q,\mu_{0})$  for $1\leq m<p_{0},\alpha\geq\Gamma$.
 Furthermore for $m\geq p_{0},\alpha\geq\Gamma$, we have
\begin{equation}\label{29}
[L_{\alpha,0},L_{0,m-1}]=(-(p+\alpha)m+p+\alpha-q\alpha)L_{\alpha,m-1}\in L(p,q,\mu_{0}).
\end{equation}
We claim the coefficient of $L_{\alpha,m-1}$ in (\ref{29}) is nonzero. Otherwise $-(p+\alpha)m+p+\alpha-q\alpha=0$,
which review as the polynomial on $m$ and so that $p+\alpha=0,p+\alpha-q\alpha=0$.
 Taking  $m=p_{0}+i$ in (\ref{29}), we have $m\geq p_{0}$ by $i\in\mathbb{Z_{+}}$. This, yields that $L_{\alpha,p_{0}+i-1}\in L(p,q,\mu_{0})$.
So, $pq=0$, which is a contradiction to $p,q\in\mathbb{C}^{*}$. The claim is proved, and so that
 $L_{\alpha,m-1}\in L(p,q,\mu_{0})$ for $m\geq1,\alpha\geq\Gamma$, namely,
\begin{equation}\label{210}
\mathcal{B}(p,q)_{[\Gamma,+\infty)}\subseteq L(p,q,\mu_{0}).
\end{equation}
Due to (\ref{24}) and (\ref{210}), $\mathcal{B}(p,q)_{[\Gamma,+\infty)}=0$.
By Lemma \ref{lem23}(\ref{43a}), $V$ has a highest weigh vector, which contradicts our assumption. Thus the map of
$\theta_{\mu_{0}}$ is injective, which implies  $\dim V_{\mu}\leq 3 \dim V_{0}+\dim V_{1}$ if $\mu\leq-2$.
 Similarly, one can derive $\dim V_{\mu}\leq  3 \dim V_{0}+\dim V_{-1}$  if $\mu\geq 2$ by Lemmas \ref{lemm21}(\ref{42b}) and \ref{lem23}(\ref{42b}).
Denote $N=\max \{3\dim V_{0}+\dim V_{1},3\dim V_{0}+\dim V_{-1}\}$. Thus $dimV_{\mu}\leq N$ for all $\mu\in\mathbb{Z}$, namely $V$ is a uniformly bounded $\mathcal{B}(p,q)$-module. The proof is completed.
\end{proof}

\section{ACKNOWLEDGMENTS}
This work is supported in part by National Natural Science Foundation of China(Grant No. 11171294), Natural Science
Foundation of Heilongjiang Province of China (Grant No. A2015007), the fund of Heilongjiang Education Committee (Grant No. 12531483).


\begin{thebibliography}{20}
\bibitem{BLP2015}C. Bai, H. Li, Y. Pei. $\phi\epsilon-$Coordinated modules for vertex algebras. Journal of Algebra, 2015, 426: 211-242.
\bibitem{BN} A. Balinskii,  S. Novikov. Poisson brackets of hydrodynamic type, Frobenius algebras and Lie algebras. Soviet Math. Dokl. 1985, 32(1): 228-231.
\bibitem{vin}  R. Block. On torsion-free abelian groups and Lie algebras. Proceedings of the American Mathematical Society, 1958, 9(4): 613-620.
\bibitem{ha} C. Boyallian, V. Kac, J. Liberati, et al. Quasifinite highest weight modules over the Lie algebra of matrix differential operators on the circle. Journal of Mathematical Physics, 1998, 39: 2910-2928.
\bibitem{Cay} A. Cayley. On the theory of analytic forms called trees, Collected mathematical papers. Cambridge Univ.
Press, Cambridge, 1890, 3: 242-246.
\bibitem{[Cay]} H. Chen, X. Guo. Unitary Harish-Chandra Modules over Block Type Lie Algebras $B(q)$ . Journal of Lie Theory, 2013, 23(3): 827-836.
\bibitem{[Ger]} H. Chen, X. Guo, K. Zhao. Irreducible quasifinite modules over a class of Lie algebras of Block type. Asian Journal of Mathematics, 2014, 18(5): 817-828.
\bibitem{GD1979} I. Gelfand, I. Dorfman. Hamiltonian operators and algebraic structures related to them. Functional Analysis and Its Applications, 1979, 13(4): 248-262.
\bibitem{Os1} V. Kac, J. Liberati. Unitary quasi-finite representations of  $W_{\infty}$ . Letters in Mathematical Physics, 2000, 53(1): 11-27.
\bibitem{hc} V. Kac, A. Radul. Quasifinite highest weight modules over the Lie algebra of differential operators on the circle. Communications in mathematical physics, 1993, 157(3): 429-457.
\bibitem{KBC} X. Kong, H. Chen, C. Bai. Classification of graded left-symmetric algebraic structures on Witt and Virasoro algebras. International Journal of Mathematics, 2011, 22(02): 201-222.
\bibitem{Mathieu} Mathieu O. Classification of Harish-Chandra modules over the Virasoro Lie algebra. Inventiones mathematicae, 1992, 107(1): 225-234.
\bibitem{bp2010} Y. Pei, C. Bai.  Realizations of conformal current-type Lie algebras. Journal of Mathematical Physics, 2010, 51(5): 052302.
\bibitem{bp2011} Y. Pei, C. Bai.   Novikov algebras and Schr$\ddot{o}$dinger-Virasoro Lie algebras. Journal of Physics A: Mathematical and Theoretical, 2011, 44(4):045201.
\bibitem{bp2012} Y. Pei, C. Bai. On an infinite-dimensional Lie algebra of Virasoro-type. Journal of Physics A: Mathematical and Theoretical, 2012, 45(23):235201.
\bibitem{Bai} Y. Su. Quasifinite representations of a Lie algebra of Block type. Journal of Algebra, 2004, 276(1): 117-128.
\bibitem{hb} Y. Su. Quasifinite representations of a family of Lie algebras of Block type. Journal of Pure and Applied Algebra, 2004, 192(1): 293-305.
\bibitem{BM2} Y. Su, C. Xia, Y. Xu. Classification of quasifinite representations of a Lie algebra related to Block type. Journal of Algebra, 2013, 393: 71-78.
\bibitem{KCB} Y. Su, C. Xia, Y. Xu. Quasifinite representations of a class of Block type Lie algebras $B (q)$ . Journal of Pure and Applied Algebra, 2012, 216(4): 923-934.
\bibitem{TangB} Y. Su. Classification of quasifinite modules over the Lie algebras of Weyl type. Advances in Mathematics, 2003, 174(1): 57-68.
\bibitem{hd} Y. Su. Quasifinite representations of some Lie algebras related to the Virasoro algebra. Recent developments in algebra and related areas, 2009, 8: 213-238.
\bibitem{tb} X. Tang, C. Bai. A class of non-graded left-symmetric algebraic structures on the Witt algebra. Mathematische Nachrichten, 2012, 285(7): 922-935.
\bibitem{su2015} Y. Su, X. Yue. Classification of {$\mathbb{Z}_2$}-graded modules of intermediate series over a Block-type Lie algebra. Communications in Contemporary Mathematics, 2015, 17(05): 1550059.
\bibitem{peib2012} Y. Pei, C. Bai. Balinsky-Novikov superalgebras and some infinite-dimensional Lie superalgebras. Journal of Algebra and Its Applications, 2012, 11(06): 1250119.
\bibitem{WCB2012} Y. Wang, Z. Chen, C. Bai. Classification of Balinsky-Novikov superalgebras with dimension $2|2$ . Journal of Physics A: Mathematical and Theoretical, 2012, 45(22): 225201.
\bibitem{BM1} Q. Wang, S. Tan. Quasifinite modules of a Lie algebra related to Block type. Journal of Pure and Applied Algebra, 2007, 211(3): 596-608.
\bibitem{YF2015} Y. Hong, F. Li. Left-symmetric conformal algebras and vertex algebras. Journal of Pure and Applied Algebra, 2015, 219(8): 3543-3567.
\bibitem{Xia2016} C. Xia. Structure of two classes of Lie superalgebras of Block type. International Journal of Mathematics, 2016, ID1650038.
\bibitem{xia2014} C. Xia, T. You, L. Zhou. Structure of a class of Lie algebras of Block type. Communications in Algebra, 2012, 40(8): 3113-3126.
\bibitem{su2007} X. Yue, Y. Su. Highest weight representations of a family of Lie algebras of Block type. Acta Mathematica Sinica, English Series, 2008, 24(4): 687-696.

\end{thebibliography}
\end{document}